\documentclass[twoside,12pt]{article}

\usepackage{times,amsmath,amssymb,amsthm,cite}

\newcommand{\subj}[1]{\par\noindent{\bf AMS Subject Classifications: }#1.}
\newcommand{\keyw}[1]{\par\noindent{\bf Keywords: }#1.}

\numberwithin{equation}{section}
\numberwithin{figure}{section}

\newtheorem{theorem}{Theorem}[section]

\theoremstyle{definition}
\newtheorem{definition}[theorem]{Definition}

\theoremstyle{remark}

\everymath{\displaystyle}
\date{}
\newcommand{\ijde}
{\vspace{-1in}\normalsize\flushleft
This is a preprint of a paper whose final and definite form will appear in\\
International Journal of Difference Equations, ISSN 0973-6069\\
{\tt http://campus.mst.edu/ijde}\\\vspace{1mm}\hrule\vspace{5mm}
\renewcommand\thefootnote{{}}

\footnotetext{\noindent\tt Received Nov 23, 2013; Revised Jan 27, 2014; Accepted Feb 15, 2014\par
\hspace*{8pt}Communicated by Agnieszka B. Malinowska}}

\usepackage{amsmath,amscd,amsfonts,amssymb,amsthm}
\usepackage{graphicx}

\usepackage{epstopdf}
\usepackage{url}

\def\a{\alpha}
\def\LHD{{_a\mathcal{D}_t^\a}}
\def\LHDHz{{_1\mathcal{D}_t^{0.5}}}

\begin{document}

\title{\ijde\center\Large\bf
A Discretization Method to Solve Fractional Variational Problems with Dependence on Hadamard Derivatives}

\author{{\bf Ricardo Almeida}\\
Center for Research and Development in Mathematics and Applications\\
Department of Mathematics, University of Aveiro\\
3810--193 Aveiro, Portugal\\
{\tt ricardo.almeida@ua.pt}\\[0.3cm]
{\bf Nuno R. O. Bastos}\\
Department of Mathematics, School of Technology\\
Polytechnic Institute of Viseu\\
3504--510 Viseu, Portugal\\
{\tt nbastos@mat.estv.ipv.pt}\\[0.3cm]
{\bf Delfim F. M. Torres}\\
Center for Research and Development in Mathematics and Applications\\
Department of Mathematics, University of Aveiro\\
3810--193 Aveiro, Portugal\\
{\tt delfim@ua.pt}}

\maketitle

% ===================================================

\thispagestyle{empty}

\begin{abstract}
We provide a fast and simple method to solve fractional variational problems
with dependence on Hadamard fractional derivatives. Using a relation between
the Hadamard fractional operator and a sum involving integer-order derivatives,
we rewrite the fractional problem into a classical optimal control problem.
The latter problem is then solved by application of standard numerical techniques.
We illustrate the procedure with an example.
\end{abstract}

\subj{26A33, 49M25}

\keyw{Hadamard's fractional calculus, fractional variational problems,
discrete approximations, optimal control}

\bibliographystyle{plain}

% ===================================================

\section{Introduction}

Fractional calculus deals with noninteger order integrals and derivatives.
In this sense, it can be seen as a generalization of ordinary calculus.
Similarly to standard (integer-order) calculus, where different notions
for integration and differentiation are available, several definitions
for fractional integrals and derivatives are also possible.
The most common fractional operators are the Riemann--Liouville and Caputo,
which are well studied in the literature. Available results
include numerical tools to handle them \cite{Almeida1}.
Here we deal with Hadamard's approach \cite{Hadamard,Kilbas2,Qian},
considering problems of the calculus of variations with dependence on this type of differentiation.

One way to find extremizers for fractional variational problems consists to deduce necessary
optimality conditions, which in this case turn out to be fractional differential equations
\cite{book:frac}. As is frequently noted, in most cases it is impossible
to find analytical solutions to such equations \cite{A:T:Leitmann,MyID:284}.
Here, instead of dealing directly with a fractional variational functional,
we approximate the fractional derivative by a sum that involves integer-order derivatives only.
In this way, we leave the field of fractional variational calculus, reducing the problem
to an ordinary optimal control problem. To solve the latter, one can use any
well-known technique to find an approximated solution. For this purpose we choose to use
a discretization process and the software package Ipopt (Interior Point OPTimizer,
\url{http://projects.coin-or.org/Ipopt}) that is designed to find solutions
to large-scale nonlinear mathematical optimization problems \cite{IPOPT}.

The text is organized as follows. In Section~\ref{sec:2}
we present a short theoretical exposition of the necessary concepts.
Our method is presented in Section~\ref{sec:3}
and then illustrated, with an example, in Section~\ref{sec:4}.
We end with Section~\ref{sec:5} of conclusion.

% ===================================================

\section{Preliminaries}
\label{sec:2}

We recall the necessary concepts and results.

\begin{definition}
Let $x:[a,b]\to\mathbb{R}$, $\a>0$, and $n=[\a]+1$.
The (left) Hadamard fractional derivative of order $\a$  is defined by
$$
\LHD x(t)=\frac{1}{\Gamma(n-\alpha)}\left(t\frac{d}{dt}\right)^n
\int_a^t \left(\ln\frac{t}{\tau}\right)^{n-\alpha-1}\frac{x(\tau)}{\tau}d\tau,
$$
where $\Gamma$ is the Gamma function, that is,
$$
\Gamma(z)=\int_0^\infty t^{z-1}e^{-t}\,dt, \quad z>0.
$$
\end{definition}

The Gamma function has the following two nice properties:
$$
\Gamma(z+1)=z\Gamma(z), \, \forall z>0
\quad \mbox{ and } \quad
\Gamma(n)=(n-1)!,\, \forall n \in \mathbb{N}.
$$
If $\a=m$ is an integer, then one has (cf. \cite{Kilbas})
$$
{_a\mathcal{D}_t^m} x(t)=\left(t\frac{d}{dt}\right)^m x(t).
$$

In \cite{Almeida} an expansion formula for the Hadamard fractional derivative
with $\a\in(0,1)$ is proved. It has the advantage that we only need the
first-order derivative of $x$ to approximate $\LHD x(t)$ by such sum.
The result is the following one.

\begin{theorem}[See \cite{Almeida}]
\label{teo}
Let $x\in C^{2}[a,b]$ and $N\geq 2$. Then,
\begin{equation*}
\LHD x(t)= {_a\mathcal{\widetilde{D}}_t^\a} x(t) + E(t),
\end{equation*}
where
$$
{_a\mathcal{\widetilde{D}}_t^\a} x(t)
= A\left(\ln\frac{t}{a}\right)^{-\a} x(t)
+B\left(\ln\frac{t}{a}\right)^{1-\a} tx^{(1)}(t)\\
+\sum_{p=2}^N C_p\left(\ln\frac{t}{a}\right)^{1-\a-p}V_p(t)
$$
with
\begin{equation*}
\begin{split}
A &=\frac{1}{\Gamma(1-\a)}\left[
1+\sum_{p=2}^N\frac{\Gamma(p+\a-1)}{\Gamma(\a)(p-1)!}\right],\\
B&= \frac{1}{\Gamma(2-\a)}\left[
1+\sum_{p=1}^N\frac{\Gamma(p+\a-1)}{\Gamma(\a-1)p!}\right],\\
C_p &= \frac{\Gamma(p+\a-1)}{\Gamma(-\a)\Gamma(1+\a)(p-1)!},
\quad p \in \{2,\ldots,N\}, \\
V_p(t) &= \int_a^t (p-1)\left(\ln\frac{\tau}{a}\right)^{p-2}
\frac{x(\tau)}{\tau}d\tau,  \quad p \in \{2,\ldots,N\},
\end{split}
\end{equation*}
and the error $|E(t)| = \left|\LHD x(t) - {_a\mathcal{\widetilde{D}}_t^\a} x(t)\right|$ is bounded by
\begin{equation}
\label{eq:bound:e}
\tilde{E}(x,t) := \max_{\tau\in[a,t]}\left|x^{(1)}(t)+tx^{(2)}(t)\right| \,
\frac{\exp((1-\a)^2+1-\a)}{\Gamma(2-\a)(1-\a)N^{1-\a}}\left(\ln\frac{t}{a}\right)^{1-\a}(t-a).
\end{equation}
\end{theorem}

In \cite{Almeida} the authors obtain a more general form of expansion,
with sums up to the $n$th derivative. Expansion formulas for
the Hadamard fractional integrals are also given.

% ===================================================

\section{Fractional Variational Problems and Discretization}
\label{sec:3}

Necessary optimality conditions for the Hadamard fractional calculus of variations are
trivially obtained as corollaries from the results of \cite{MR2989364,MyID:268}.
However, in general, there are no analytical methods to solve
such fractional differential equations. Therefore, to establish simple
numerical methods is an important issue. Here
we provide a direct method to address such variational problems.
Our idea is simple but effective and easy to implement.

The problem of the fractional calculus of variations
under consideration is stated in the following way:
minimize the functional
\begin{equation}
\label{eq:J}
J(x)=\int_a^b L\left(t,x(t),\LHD x(t)\right)\,dt
\end{equation}
subject to the boundary conditions
$$
x(a)=x_a \quad \mbox{and} \quad x(b)=x_b,
$$
where $x_a$ and $x_b$ are two given reals and $0 < \alpha < 1$.
For other types of fractional variational problems
we refer the reader to \cite{book:frac}.
Note that functions $x$ should be absolutely continuous in order
to guarantee the existence of the Hadamard fractional derivatives
\cite[Lemma 2.34]{Kilbas}. To apply our method,
more smoothness is, however, required: admissible functions
are assumed to be $C^2$, in agreement with Theorem~\ref{teo}.
To solve the problem, we first replace
the fractional operator by the sum given in Theorem~\ref{teo}.
By doing so, we get a classical optimal control problem: minimize the functional
\begin{multline}
\label{eq:Japprox}
J_{approx}(x,u,V_2,\ldots,V_n)
=\int_a^b L\left(t,x(t),A\left(\ln\frac{t}{a}\right)^{-\a} x(t)\right.\\
\left.+B\left(\ln\frac{t}{a}\right)^{1-\a} tu(t)
+\sum_{p=2}^N C_p\left(\ln\frac{t}{a}\right)^{1-\a-p}V_p(t)\right)\,dt
\end{multline}
subject to the dynamic constraints
$$
\left\{\begin{array}{l}
x'(t)=u(t),\\
V_p'(t)=(p-1)\left(\ln\frac{t}{a}\right)^{p-2}\frac{x(t)}{t},
\quad p \in \{2,\ldots,N\},
\end{array}\right.
$$
and the boundary conditions
$$
\left\{\begin{array}{l}
x(a)=x_a,\\
x(b)=x_b,\\
V_p(a)=0,  \quad p \in \{2,\ldots,N\}.
\end{array}\right.
$$
This is a classical optimal control problem with state functions $x$ and $V_p$,
$p \in \{2,\ldots,N\}$, and control $u$. We solve it
applying any standard discretization procedure \cite{Betts}.

The error of approximating an Hadamard derivative is given by Theorem~\ref{teo}.
Using the bound for the error \eqref{eq:bound:e}, it is possible to get a bound for
the error of approximating the variational functional \eqref{eq:J} by the approximated value
\eqref{eq:Japprox}.

\begin{theorem}
Let $x\in C^{2}[a,b]$ and $N\geq 2$.
The error of approximating
the variational functional \eqref{eq:J}
by \eqref{eq:Japprox}
is bounded as follows:
\begin{equation*}
\left|J(x) - J_{approx}(x,u,V_2,\ldots,V_n)\right| \le M \int_a^b \tilde{E}(x,t) dt,
\end{equation*}
where
$$
M := \max_{t \in [a,b]} \left|\partial_3 L\left(t,x(t),\LHD x(t)\right)\right|
$$
and $\tilde{E}(x,t)$ is given by \eqref{eq:bound:e}.
\end{theorem}

\begin{proof}
It follows from Taylor's theorem and Theorem~\ref{teo}.
\end{proof}

% ===================================================

\section{An Example}
\label{sec:4}

Let us minimize the functional
\begin{equation}
\label{ex1}
J(x)=\int_1^2 \left(\LHDHz x(t)-\frac{\sqrt{\ln t}}{\Gamma(1.5)}\right)^2\,dt
\end{equation}
subject to the boundary conditions
\begin{equation}
\label{ex1:bc}
x(1)=0 \quad \mbox{and} \quad x(2)=\ln 2.
\end{equation}
It is easy to check that the (global) solution is $\overline{x}(t)=\ln t$. Indeed,
$$
\LHDHz \overline{x}(t)=\frac{\sqrt{\ln t}}{\Gamma(1.5)}
$$
and the functional \eqref{ex1} is nonnegative with $J\left(\overline{x}\right) = 0$.
Using the approach explained in Section~\ref{sec:3}, we approximate this problem
by the following one: minimize the functional
\begin{multline}
\label{ex2}
J_{approx}(x,u,V_2,\ldots,V_N)
=\int_1^2 \left(A\left(\ln t\right)^{-0.5} x(t)
+B\left(\ln t\right)^{0.5} tu(t)\right.\\
\left.+\sum_{p=2}^N C_p\left(\ln t\right)^{0.5-p}V_p(t)
-\frac{\sqrt{\ln t}}{\Gamma(1.5)}\right)^2\,dt
\end{multline}
subject to the dynamic constraints
\begin{equation}
\label{eq:dc}
\left\{\begin{array}{l}
x'(t)=u(t),\\
V_p'(t)=(p-1)\left(\ln t\right)^{p-2}\frac{x(t)}{t},  \quad p \in \{2,\ldots,N\},
\end{array}\right.
\end{equation}
and the boundary conditions
\begin{equation}
\label{eq:bc}
\left\{\begin{array}{l}
x(1)=0,\\
x(2)=\ln 2,\\
V_p(1)=0,  \quad p \in \{2,\ldots,N\}.
\end{array}\right.
\end{equation}
Fix $N=3$ and $k\in \mathbb{N}$ and consider the maximum error given by
$$
E_k(x,\tilde x)= \max_{i\in\{1,\ldots,k\}} | x(t_i)-\tilde x(t_i) |.
$$
We apply a direct discretization process with $k$ points to problem \eqref{ex2}--\eqref{eq:bc},
and then we use Ipopt \cite{IPOPT} to find a solution for it. In Figure~\ref{state} we show the results
for the state function $x$ with $k=100$ and $k=250$. From these results,
we obtain the errors $E_{100}(x,\tilde x)=0.029558802$ and $E_{250}(x,\tilde x)=0.011928571$,
and we can see that the error decreases as the number of points $k$ increases.
\begin{figure}[!ht]
 \centering
\includegraphics[scale=0.3]{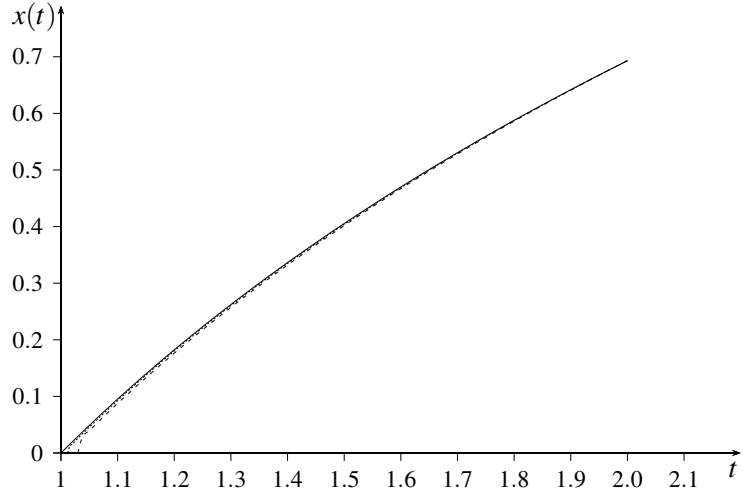}
\caption{Exact solution to \eqref{ex1}--\eqref{ex1:bc} (continuous line)
versus numerical approximations to \eqref{ex2}--\eqref{eq:bc}
with $k=250$ (dotted line) and $k=100$ (dashed line).}
\label{state}
\end{figure}
As is shown in Figure~\ref{state}, when $k=250$
the approximation is almost indistinguishable from the exact solution.

% ===================================================

\section{Conclusion}
\label{sec:5}

The Hadamard fractional operator was introduced by J. Hadamard in 1892 \cite{Hadamard}.
This notion is different from the better known Riemann--Liouville and Caputo derivatives,
in the sense that the kernel of the integral (in the definition of Hadamard
derivative) contains a logarithmic function of arbitrary exponent.
Recently, there has been an increasing interest in studying
variational problems with generalized fractional operators, that is,
with fractional operators depending on a general kernel --- see, e.g.,
\cite{MR2989364,MyID:268}. Such generalized theory includes, as a particular case,
the Hadamard variational calculus. However, available results are,
excluding the particular cases of Riemann--Liouville and Caputo,
only analytical. Here we investigated nonstandard fractional variational problems
depending on Hadamard fractional derivatives.
Relation between the Hadamard fractional operator and a sum involving integer-order
derivatives is used to rewrite the fractional problem into a classical optimal control problem.
The latter problem is then solved by application of standard numerical techniques.
The effectiveness of the numerical approach is illustrated with an example.

% ===================================================

\section*{Acknowledgments}

This article was supported by Portuguese funds through the
\emph{Center for Research and Development in Mathematics and Applications} (CIDMA),
and \emph{The Portuguese Foundation for Science and Technology} (FCT),
within project PEst-OE/MAT/UI4106/2014.
The authors are grateful to two anonymous referees
for valuable remarks and comments, which
significantly contributed to the quality of the paper.

% ===================================================

% ===================================================

\end{document}